\documentclass[12pt]{amsart}

\usepackage[all]{xy}

\newtheorem{theorem}{Theorem}[section]
\newtheorem{lemma}[theorem]{Lemma}
\newtheorem{corollary}[theorem]{Corollary}
\newtheorem{proposition}[theorem]{Proposition}
\newtheorem{remark}[theorem]{Remark}
\newtheorem{definition}[theorem]{Definition}

\newcommand{\ncom}{\newcommand}
\ncom{\ep}{\epsilon}
\ncom{\rar}{\rightarrow}
\ncom{\thrar}{\twoheadrightarrow}
\ncom{\lrar}{\longrightarrow}
\ncom{\ov}{\overline}
\ncom{\what}{\widehat}
\ncom{\dar}{\downarrow}    
\ncom{\ery}{\eqnarray}
\ncom{\xrar}{\xrightarrow}
\ncom{\xdar}{\xdownarrow}
\usepackage{subfig}
\usepackage{amscd}

\newcommand{\ignore}[1]{}
\ncom{\m}{\mbox}
\ncom{\sta}{\stackrel}
\ncom{\C}{{\mathbb C}}
\ncom{\A}{{\mathbb A}}
\ncom{\Z}{{\mathbb Z}}
\ncom{\Q}{{\mathbb Q}}
\ncom{\R}{{\mathbb R}}
\ncom{\G}{{\mathbb G}}
\ncom{\HH}{{\mathbb H}}
\ncom{\al}{\alpha}
\ncom{\p}{{\mathbb P}}
\ncom{\N}{{\mathbb N}}
\ncom{\K}{{\mathbb K}}
\ncom{\X}{{\mathbb X}}
\ncom{\f}{\frac}
\ncom{\cA}{{\mathcal A}}
\ncom{\cB}{{\mathcal B}}
\ncom{\cD}{{\mathcal D}}
\ncom{\cDB}{{\mathcal D \mathcal B}}
\ncom{\cX}{{\mathcal X}}
\ncom{\cO}{{\mathcal O}}
\ncom{\cW}{{\mathcal W}}
\ncom{\cL}{{\mathcal L}}
\ncom{\cP}{{\mathcal P}}
\ncom{\cH}{{\mathcal H}}
\ncom{\cS}{{\mathcal S}}
\ncom{\cM}{{\mathcal M}}
\ncom{\cC}{{\mathcal C}}
\ncom{\cT}{{\mathcal T}}
\ncom{\cF}{{\mathcal F}}
\ncom{\cN}{{\mathcal N}}
\ncom{\cJ}{{\mathcal J}}
\ncom{\cV}{{\mathcal V}}
\ncom{\cZ}{{\mathcal Z}}
\ncom{\cU}{{\mathcal U}}
\ncom{\cSU}{{\mathcal S \mathcal U}}
\ncom{\cG}{{\mathcal G}}
\ncom{\cQ}{{\mathcal Q}}
\ncom{\cR}{{\mathcal R}}
\ncom{\cY}{{\mathcal Y}}
\ncom{\cE}{{\mathcal E}}
\ncom{\cI}{{\mathcal I}}
\ncom{\mylabel}[1]{{\rm (#1)}\label{#1}}
\ncom{\Hom}{{\textit{Hom}}}
\ncom{\eop}{{\hfill $\Box$}}
\begin{document}
\baselineskip=16pt

\title[]{Embedding properties of linear series on hyperelliptic varieties}


\author[S.Chintapalli]{Seshadri Chintapalli}
\author[J. N. Iyer]{Jaya NN  Iyer}

\address{The Institute of Mathematical Sciences, CIT
Campus, Taramani, Chennai 600113, India}
\email{seshadrich@imsc.res.in}
\email{jniyer@imsc.res.in}

\footnotetext{Mathematics Classification Number: 53C55, 53C07, 53C29, 53.50. }

\begin{abstract}
In this paper, we investigate linear systems on hyperelliptic varieties. We prove analogues of well-known theorems on abelian varieties, like Lefschetz's embedding theorem and higher k-jet embedding theorems.
Syzygy or $N_p$-properties are also deduced for appropriate powers of ample line bundles.
\end{abstract}
\maketitle

\setcounter{tocdepth}{1}
\tableofcontents



\section{Introduction}

Suppose $L$ is an ample line bundle on a smooth projective variety $X$. Some questions that arise are basepoint freeness, very ampleness, and syzygy properties or $N_p$-properties, $p\geq 0$, associated to the line bundle $L$ on $X$.
These properties are fairly well-understood on curves, surfaces and abelian varieties \cite{Green}, \cite{Voisin}, \cite{Gallego}, \cite{Lefschetz}, \cite{Kempf}, \cite{Iyer},\cite{Pareschi}, \cite{PP}. There are conjectures by Fujita and Mukai \cite[Conjecture 4.2]{EL} on the behaviour of (adjoint) linear systems $|K_X+L^{\otimes r}|$, associated to powers of ample line bundles tensored with the canonical line bundle $K_X$ of $X$.  

The aim of this paper is to investigate above questions for linear systems on hyperelliptic varieties. A hyperelliptic surface $S$ is a complex projective surface which is not an abelian surface,  but admitting an \'etale cover $ A\rar S$, where $A$ is an abelian surface.
Hyperelliptic surfaces were classified by Enriques-Severi and Bagnera-de Franchis \cite{En-Se}, \cite{Bg-dF}. More generally, H. Lange \cite{Lange} extended this notion to higher dimensions. A smooth projective variety $X$ is called a \textit{hyperelliptic variety} if it is not isomorphic to an abelian variety but admitting an \'etale covering $A\rar X$, where $A$ is an abelian variety.

We show in \S \ref{veryample}, the following analogue of Lefschetz embedding:

\begin{theorem}
Suppose $X$ is a hyperelliptic variety of dimension $n$. Let $L$ be an ample line bundle on $X$. Then we have

1) $L^k$, for $k\geq 3$, is always very ample. 

2) $L^2$ is very ample, if $L$ has no base divisor.
\end{theorem}

These are well-known theorems on abelian varieties due to Lefschetz, Kempf and Ohbuchi \cite{Lefschetz}, \cite{Kempf}, \cite{Ohbuchi}.
Furthermore, we extend generalizations of above concepts, namely $k$-jet ampleness, to hyperelliptic varieties, as follows.

\begin{theorem}
Suppose $L$ is an ample line bundle on a hyperelliptic variety $X$. Then the following hold, for $k\geq 0$:

1) $L^{k+2}$ is $k$-jet ample

2) $L^{k+1}$ is $k$-jet ample if $L$ has no base divisor.
\end{theorem}

Regarding $N_p$-property, we show the analogue of Pareschi's theorem \cite{Pareschi}(Lazarsfeld's conjecture) on abelian varieties, extended to hyperelliptic varieties.

\begin{theorem}
Suppose $L$ is an ample line bundle on a hyperelliptic variety $X$. Then $L^{p+k}$ satisfies $N_p$-property, for $k\geq 3$.
\end{theorem}

The key point in the proof is to note that a hyperelliptic variety $X$ is realized as a finite group quotient $A/G$ of an abelian variety $A$, for some finite group $G$ acting freely on $A$ \cite[Theorem 1.1, p.492]{Lange}. Hence a line bundle on a hyperelliptic variety is regarded as a $G$-linearized line bundle on $A$. 
We introduce the notion of $G$-global generation of $G$-linearized sheaves in \S \ref{Gglobal} and obtain correspondence of usual global generation on $X$ with $G$-global generation on $A$. 
We then look at the notion of $M$-regularity of $G$-linearized sheaves and suitably extend the techniques used by Pareschi and Popa \cite{PP},\cite{PP2}. The proofs are reduced to showing $G$-global generation of appropriate $G$-linearized coherent sheaves, obtained by applying the Fourier-Mukai functor.

We note that in \cite{PP}, a part of above results are obtained for irregular varieties. It is known that hyperelliptic surfaces are irregular, 
however higher dimensional hyperelliptic varieties may not always be irregular. Hence, the 'averaging' of sections employed in \S \ref{averaging}
is relevant and new, and is used to descend data suitably.
 
The Syzygy property is proved for powers of ample line bundles in \S \ref{syzygy}. It is implied by  vanishing of the first cohomology of the bundle $M_L$ (the kernel of the evaluation map), twisted with appropriate powers of an ample line bundle $L$. 

There are further questions,  determining the equations of special subvarieties $\cW_d^r$ (parametrising degree $d$ line bundles on $C$ with at least $r+1$ sections) inside $Jac(C)$ and on finite quotients of $Jac(C)$. One also needs to understand $M$-regularity of the theta divisor restricted to these subvarieties.
We hope to look into it in future.

\section{$G$-linearized sheaves and Fourier-Mukai functor}
 

Suppose $X$ is a hyperelliptic variety of dimension $n$ defined over the complex numbers. By definition, it is not an abelian variety but it admits an \'etale cover $A\rar X$ such that $A$ is an abelian variety. By \cite[Theorem 1.1, p.492]{Lange}, there is a finite group $G$ acting biholomorphically on $A$, without fixed points.
In other words, we can write $X$ as a group quotient $X=A/G$, with an \'etale quotient morphism
$$
\pi: A\rar X=A/G.
$$
To investigate coherent sheaves on $X$, we note that their pullback on $A$ under the morphism $\pi$, is equipped with an action of the group $G$.
Hence to investigate line bundles and more generally coherent sheaves on $X$, it would suffice to investigate coherent sheaves on $A$ with a $G$-action.
To make this more precise, we recall the following facts.

\subsection{$G$-linearized sheaves}\cite{MFK}

 Suppose $A$ is an abelian variety and is equipped with an action by a finite group $G$.
 In this subsection, we recall $G$-linearized sheaves on an abelian variety $A$.  
 \begin{definition}
\cite[Definition 1.6, p.30]{MFK}. A coherent sheaf $\mathcal{F}$ on $A$ is called $G$-linearized (or a $G$-sheaf) if we have an isomorphism $\phi_g:{g^*{\mathcal{F}}\xrightarrow{\sim}\mathcal{F} }$, 
for all $g\in G$, and such that following diagram of coherent sheaves on $A$

\begin{equation*}
\xymatrix{(gh)^*\cF \ar[r]^{h^*\phi_g} \ar[dr]_{\phi_{gh}}&  h^*{\cF} \ar[d]^{\phi_h}\\ & \cF}
\end{equation*}
is commutative, for any pair $g,h\in G$, i.e. $\phi_{gh}=\phi_h\circ h^*{\phi_g}$.
\end{definition} 

Assume that the action of the group $G$ on $A$ is free.
We note that $G$-linearized sheaves are relevant to our situation, since it corresponds to coherent sheaves on the quotient variety $A/G$.
In fact, we have:

\begin{proposition}\label{oneone}
Consider a pair $(A,G)$ as above, and assume that the action of $G$ on $A$ is free. Then the functor $\cF\mapsto \pi^*\cF$ is an equivalence of category of coherent $\cO_X$-modules on $X$ and the category of coherent $G$-sheaves on $A$. The inverse functor is given by $\cG \mapsto (\pi_*(\cG))^G$ (the subsheaf of $G$-invariant sections of $\pi_*(\cG)$). Locally free sheaves correspond to locally free sheaves of the same rank.
\end{proposition}
\begin{proof}
See \cite[Proposition 2, p.70]{Mumford}.
\end{proof} 
 
 We will use this proposition when we define $G$-global generation of $G$-coherent sheaf on $A$, in \S \ref{Gglobal}, and its equivalence with the usual global generation of the corresponding sheaf on the quotient variety $X$.

 \subsection{Fourier-Mukai functor}
 
 Suppose $A$ is an abelian variety of dimension $g$ over $\mathbb C$ and $\hat{A}$ be its dual abelian variety  \cite[Section 2.4, p.34]{BL}. Denote $\cP$ the normalized Poincar$\acute{e}$ line bundle on $A\times\hat{A}$. Let $G$ be a finite group acting on $A$ and
$\pi:A\rightarrow X=A/G$ be the quotient morphism. 


We recall some facts from \cite{M}.
Denote $\cC{oh}(A)$ (respectively $\cC{oh}(\hat{A})$) the category of coherent sheaves on $A$ (resp. on $\hat{A}$).
Let 
$$
\mathcal{\hat{S}}: \cC{oh}(A) \rightarrow \cC{oh}(\hat{A})
$$ 
be the functor defined as follows: 
$$
\mathcal{\hat{S}}\cF:={p_2}_*(p^*_1\cF\otimes\cP).
$$
Similarly we can define the functor
$$
\cS:\cC{oh}(\hat{A})\rar \cC{oh}(A)
$$ 
given as 
$$
\mathcal{S}\cG:={p_1}_*(p^*_2\cG\otimes\cP).
$$
Denote $D(A)$ (respectively $D(\hat{A})$) the derived category of $\cC{oh}(A)$ (respectively $\cC{oh}(\hat{A})$). Using \cite[Proposition 2.1, p.155]{M}, we have a
derived functor 
$$
\cR\mathcal{\hat{S}}:\cD(A)\rightarrow \cD(\hat{A})
$$ 
given by 
$$
\cR\mathcal{\hat{S}}\cF=\cR{p_2}_*(p^*_1\cF\otimes\cP).
$$
Similarly we obtain the derived functor
$$
\cR\cS:\cD(\hat{A})\rar \cD(A).
$$ 
The derived functors are called the Fourier-Mukai functor.

Let $\cD^G(A)$ 
be the subcategory of complex of coherent $G$-sheaves, of $\cD(A)$. Then the Fourier functor $\cR\mathcal{\hat{S}}$ restricts to a functor on $\cD^G(A)$.

\subsection{Mukai-regularity}

Recall the notion of of I.T (index theorem) and M-regularity from \cite{M} and \cite{PP}. In particular we state them for coherent $G$-sheaves.

With notations as in previous subsection, denote $R^j\hat{\cS}(\cF)$, the cohomologies of the derived complex $R\hat{\cS}\cF$.
A coherent $G$-sheaf $\cF$ on $A$ satisfies W.I.T (the \textit{weak index theorem}) with index $i$ if $R^j\hat{\cS}(\cF)=0$, for all $j\neq i$.

A stronger notion is as below.

\begin{definition}
A coherent $G$-sheaf $\mathcal{F}$ on $X$ is said to satisfy I.T (index theorem) with index $i$ if $H^j(\mathcal{F}\otimes \alpha)=0$, for all $\alpha\in \hat{A}$
and for all $j\neq i$.
\end{definition}

In this situation the sheaf $R^i\hat{\cS}(\cF)$ is locally free. If $\cF$ satisfies W.I.T or I.T. with index $i$, then the sheaf $R^i\hat{\cS}(\cF)$ is denoted by $\hat{\cF}$ and is called the \textit{Fourier transform} of $\cF$.

Given a coherent sheaf $\cF$ on $A$, we denote the support of the sheaf $R^i\hat{\cS}(\cF)$ by
$$
S^i(\cF):= \m{Supp}(R^i\hat{\cS}(\cF)).
$$
\begin{definition}
A coherent $G$-sheaf $\mathcal{F}$ on $A$ is called M-regular if 
$$
\m{codim } S^i(\mathcal{F})\,>\,i
$$ 
for each $i=1,...,g$.

\end{definition}

\begin{remark}\label{AM}
1) Coherent $G$-sheaves on $X$ which satisfy I.T with index $0$, are examples of M-regular $G$-sheaves.

2) We also note that an ample line bundle $H$ satisfies I.T with index $0$  \cite[Example 2.2, p.289]{PP}. This will be relevant in our later sections.
\end{remark}

Denote the cohomological support locus \cite{Green2}:
$$
V^i(\cF):= \{\eta\in \m{Pic}^0(A): h^i(\cF\otimes \eta)\neq 0 \} \subset \m{Pic}^0(A).
$$
There is an inclusion $S^i(\cF)\subset V^i(\cF)$.

Hence a $G$-sheaf is $M$-regular if 
\begin{equation}\label{cohlocus}
\m{codim}(V^i(\cF))\,>\,i
\end{equation}
for any $i=1,...,g$.

\section{$G$-global generation and global generation on hyperelliptic varieties}\label{Gglobal}


Suppose $|G|=k$ and the group $G$ acts on a coherent sheaf $\cF$.
Consider the extension of $G$ by $\mu_k$, the group of $k^{th}$ roots of unity. In other words, there is an exact sequence:
$$
1 \rightarrow \mu_k \rightarrow \tilde{\cG}\rightarrow G \rightarrow 0.
$$
We assume that there is a splitting and let $\tilde{G}\subset \tilde{\cG}$ denote the image of $G$ under the splitting map.

Then we note that $\tilde{G}$ acts on $H^0(A, \cF)$.
Denote the subspace of $\tilde{G}$-invariants: 
$$
H^0(A,\cF)^{\tilde{G}}= {\lbrace s\in H^0(A,\cF):\tilde{g}{{s}=s \,\,\forall \tilde{g} \in \tilde{G} \rbrace}}.
$$

Since our aim is to to obtain global generation of coherent sheaves on the quotient variety $X=A/G$, we introduce the following notion for coherent $G$-sheaves on $A$ as follows. In the next subsection, we will prove its equivalence with usual global generation on $X$.

\subsection{$G$-global generation, $G$-very ampleness and $G$-$k$ jet ampleness}

We keep notations as above.

\begin{definition}
A coherent $G$-sheaf $\mathcal{F}$ on $A$ is called $G$-globally generated if the evaluation map
$$
ev:{H^0(A,\mathcal{F})}^{\tilde{G}}\otimes\mathcal{O}_A \rightarrow\mathcal{F}
$$
is surjective. Here the map $ev$ is evaluation of $\tilde{G}$-invariant sections at any point of $A$.
\end{definition}

Now we formulate very ampleness for coherent $G$-sheaves as follows. For any $a\in A$, let
$G.a:=\{ga: \m{for any} g\in G\}$. Then this is the orbit of the point $a\in A$ under the action of $G$. Let $I_{G.a}$ denote the ideal sheaf of the orbit $G.a$ in $A$. Then this is a coherent $G$-sheaf on $A$.

\begin{definition}
A $G$-line bundle $L$ on $A$ is called $G$-very ample if the coherent $G$- sheaf $L\otimes I_{G.a}$ is $G$-globally generated, for all $a\in A$.
\end{definition}

This notion can be extended to $k$-jet very amplessness for $G$-line bundles as well. We do it as follows.

\begin{definition}
A $G$-line bundle $L$ on $A$ is $G$-$k$-jet very ample if the coherent $G$ sheaf
$$
L\otimes I_{G.a_1}^{k_1}\otimes ...\otimes I_{G.a_l}^{k_l}
$$
is $G$-globally generated, for distinct points $a_1,a_2,...,a_l\in A$ such that $k_1+k_2+...+k_l=k$.
In other words, the evaluation map given by $\tilde{G}$-invariant sections
$$
H^0(A,L\otimes I_{G.a_1}^{k_1}\otimes ...\otimes I_{G.a_l}^{k_l})^{\tilde{G}} \rar H^0(A,L\otimes I_{G.a_1}^{k_1}\otimes ...\otimes I_{G.a_l}^{k_l}\otimes \cO_A/m_a)
$$
is surjective, for each $a\in A$.

\end{definition}

Note that $G$-$0$-jet ample is same as $G$-global generation and $G$-$1$-jet ampleness is same as $G$-very ampleness.

\subsection{Equivalence of $G$-global generation and  global generation on $X=A/G$}

In this subsection, we note the relevance of $G$-global generation on the quotient variety $X$. We keep notations as in the previous subsection.

Then we have the following equivalence:

\begin{lemma}\label{GGoneone}
Suppose $\cF$ is a coherent $G$-sheaf on $A$. Then $\cF$ is $G$-globally generated if and only if the corresponding sheaf $(\pi_*(\cF))^{G}$ is globally generated on the quotient variety $X=A/G$. 
\end{lemma}
\begin{proof}
We recall the one-one correspondence of coherent sheaves, as given in Proposition \ref{oneone}. Given a coherent sheaf $\cG$ on the quotient variety $X=A/G$, consider its pullback $\pi^*\cG$ on $A$, via the quotient morphism $\pi:A\rar X=A/G$.
Then $\pi^*\cG$ is a coherent $G$-sheaf on $A$. It would suffice to prove that $\cG$ is globally generated on $X$ if and only if $\pi^*\cG$ is $G$-globally generated on $A$, using Proposition \ref{oneone}.

Firstly, we note the following decomposition:
$$
\pi_*\cO_A= \bigoplus_{\chi\in \hat{G}} L_\chi.
$$
Here $L_\chi$ is a line bundle on $X$ associated to the character $\chi$ on $G$. Using projection formula, we have:
\begin{equation}\label{directsum}
\pi_*(\pi^*\cG)= \bigoplus_{\chi\in \hat{G}}\cG\otimes L_\chi.
\end{equation}
This gives us a a decomposition of the space of global sections:
$$
H^0(A,\pi^*\cG)= \bigoplus_{\chi\in \hat{G}} \pi^* H^0(X, \cG\otimes L_\chi).
$$
In particular the subspace of $\tilde{G}$-invariant sections of $H^0(A,\pi^*\cG)$ is given by $\pi^*H^0(X,\cG)$.  

Suppose $\cG$ is globally generated. This implies that the evaluation map:
$$
H^0(X,\cG)\otimes \cO_X \rar \cG
$$
is surjective.
The pullback of this morphism of sheaves, via $\pi$, on $A$ corresponds to
$$
H^0(A,\pi^*\cG)^{\tilde{G}}\otimes \cO_A \rar \pi^*\cG
$$
and which is clearly surjective. This implies the $G$-global generation of $\pi^*\cG$. Using the equivalence of categories in Proposition \ref{oneone}, we conclude the proof.

\end{proof}

\begin{corollary}\label{GGample}
Suppose $L$ is an ample $G$-line bundle on $A$ and $M$ be the corresponding line bundle on $X$ (under the correspondence in Proposition \ref{oneone}).
Then $L$ is $G$-$k$ jet ample if and only if $M$ is $k$-jet ample on $X$.
\end{corollary}
\begin{proof}
We need only to note that the ideal sheaf $I_{x_1}^{k_1}\otimes...\otimes I_{x_l}^{k_l}$ of distinct points $x_1,...,x_l\in X$ with multiplicities $k_i$, such that $\sum_ik_i=k$, corresponds to the ideal sheaf $I_{G.a_1}^{k_1}\otimes...\otimes I_{G.a_l}^{k_l}$ on $A$, under the correspondence in Proposition \ref{oneone}. Here $G.a_i=\pi^{-1}(x_i)$, i.e. the inverse image of a point $x_i$ is a $G$-orbit of a point $a_i\in A$, for $i=1,...,l$.  
Hence the coherent $G$-sheaf $L\otimes I_{G.a_1}^{k_1}\otimes...\otimes I_{G.a_l}^{k_l}$ on $A$ corresponds to the coherent sheaf $M\otimes I_{x_1}^{k_1}\otimes...\otimes I_{x_l}^{k_l}$ on $X$. Now we apply Lemma \ref{GGoneone}, to conclude the proof.
\end{proof}

\section{$G$-global generation of $G$-linearized sheaves of weak index zero}\label{GGGG}


In this section, we recall the notion of continuous global generation \cite{PP}, adapted to coherent $G$- sheaves. Instead of the usual multiplication maps, we take the 'averaging' of sections, for the action of the group $G$. We note that the results of this section hold, for any action of the finite group, i.e., the action need not be free, except in Proposition \ref{CGG}.

Before proceeding to continuous global generation and its relevance to our set-up, recall the surjectivity statement for multiplication map of sections of ample line bundles \cite[7.3.3]{BL}. This is suitably generalized to higher rank sheaves, which are $M$-regular, by Pareschi and Popa \cite{PP}.
We modify the multiplication maps by taking 'averaging' of sections, for the finite group $G$. In other words, we will consider multiplication maps for the $\tilde{G}$-invariant sections, suitably interpreted. This will be needed when we want to look at $G$-global generation of coherent $G$ sheaves.

\subsection{Surjectivity of 'Averaging' map}\label{averaging}

We keep the notations from the previous section.

\begin{lemma}\label{KT}
Let $\cF$ be M-regular coherent $G$-sheaf and $H$ locally free $G$-sheaf satisfying I.T with index $0$. Then
for any Zariski open set $U\subseteq\hat{A}$, the map
$$
\bigoplus_{\alpha\in U}{H^0(\mathcal{F}\otimes\alpha)\otimes H^0(H\otimes\check{\alpha})}\xrightarrow{\oplus{Av}} H^0(\mathcal{F}\otimes H)^{\tilde{G}}
$$
is surjective. Here the 'averaging map' is given as
 $$
 Av(s\otimes t)= \frac{1}{|G|}{\sum_{\tilde{g}\in \tilde{G}}{\tilde{g}(s\otimes t)}},
 $$
 for $s\in H^0(\cF\otimes\alpha)$ and $t\in H^0(H\otimes \check{\alpha})$.
\end{lemma}
\begin{proof}

Firstly, note that the map $\oplus Av$ factorizes as follows,
$$
\bigoplus_{\alpha\in U}{H^0(\mathcal{F}\otimes\alpha)\otimes H^0(H\otimes\check{\alpha})} \xrightarrow {\sum{m_\alpha}} H^0(\cF\otimes H) \xrightarrow {h} H^0(\cF\otimes H)^{\tilde{G}}.
$$
where $h$ is the averaging map. 
By \cite[Theorem 2.5, p.290]{PP}, the map $\sum{m_\alpha}$ is surjective. Clearly $h$ is surjective, since $h$ restricts to identity on $ H^0(\cF\otimes H)^{\tilde{G}} \subset H^0(\cF\otimes H)$.
Hence the composed map $\oplus Av=h\circ \sum{m_\alpha}$ is surjective.

\end{proof}

\begin{corollary}\label{KTC}
Let $\cF$ be M-regular coherent $G$-sheaf and $H$ locally free $G$-sheaf satisfying I.T with index $0$. Then
for any large positive integer $N$ and for any subset $S\subset\hat{X}$ with $|S|=N$, the averaging map
$$
\bigoplus_{\alpha\in S}{H^0(\mathcal{F}\otimes\alpha)\otimes H^0(H\otimes\check{\alpha})}\xrightarrow{\oplus{Av}} H^0(\mathcal{F}\otimes H)^{\tilde{G}}
$$
is surjective
\end{corollary}
\begin{proof}
By above Lemma \ref{KT}, the surjectivity of the averaging map
$$
\bigoplus_{\alpha\in U}{H^0(\mathcal{F}\otimes\alpha)\otimes H^0(H\otimes\check{\alpha})}\xrightarrow{\oplus{Av}} H^0(\mathcal{F}\otimes H)^{\tilde{G}}
$$
implies that the family of linear suspaces ${\lbrace Im(Av_\alpha)\rbrace}_{\alpha\in U}$ spans the finite dimensional vector space $ H^0(\mathcal{F}\otimes H)^{\tilde{G}}$.
So for any large positive integer $N$, the images under $Av$ of a finitely many $N$ linear subspaces $H^0(\mathcal{F}\otimes\alpha)\otimes H^0(H\otimes\check{\alpha})$
span $ H^0(\mathcal{F}\otimes H)^{\tilde{G}}$.

\end{proof}

\subsection{$G$-Continuous Global Generation}

In this subsection, we recall the notion of continuous global generation and its relevance to global generation \cite{PP}.
We suitably modify this notion for coherent $G$-sheaves and show that it is related to $G$-global generation.

\begin{definition}
A coherent $G$-sheaf $\cF$on $X$ is called $G$-continuously globally generated if for any nonempty open set $U\subseteq\hat{X}$ the sum of average maps
$$
\bigoplus_{\alpha\in U}{H^0(\mathcal{F}\otimes\alpha)\otimes\check{\alpha}}\xrightarrow{\oplus{Av}}\mathcal{F}
$$ 
is surjective. For $s\in H^0(A, \cF\otimes \al)$ and a local section $t$ of $\check\al$, we define locally on $A$: 
$$
Av(s\otimes t)=\frac{1}{|G|}{\sum_{\tilde{g}\in \tilde{G}}{\tilde{g}.(s\otimes t)}}.
$$
Note that locally $s\otimes t$ is a section of $\cF$.
\end{definition}

As earlier, we note that the sum could be taken over finite subsets of $\m{Pic}^0(A)$, of large cardinality.

\begin{lemma}\label{L1}
Suppose $\cF$ is a coherent $G$-sheaf and assume it is $G$-continuously globally generated. Then for any large positive integer $N$
and for any subset $S\subset \hat{X}$ with $|S|=N$, the sum of average maps
$$
\bigoplus_{\alpha\in S}{H^0(\mathcal{F}\otimes\alpha)\otimes\check{\alpha}}\xrightarrow{\oplus{Av}}\mathcal{F}
$$ 
is surjective.
\end{lemma}
\begin{proof}
This proof is similar to the argument given in Corollary \ref{KTC}.
\end{proof}

We now prove the following proposition relating tensor product of continuously $G$ global generated sheaves and $G$-global generation.

\begin{proposition}\label{GG}
Suppose $\mathcal{F}$ is a coherent $G$-sheaf and $H$ is a $G$-line bundle on $X$. If both $\mathcal{F}$ and $H$ are $G$-continuously globally generated
then  $\mathcal{F}\otimes H$ is $G$-globally generated.
\end{proposition}
\begin{proof}
By Lemma \ref{L1}, for any large positive integer $N$
and for any subset $S\subset \hat{A}$ with $|S|=N$, the averaging map 
$$
\bigoplus_{\alpha\in S}{H^0(\mathcal{F}\otimes\alpha)\otimes\check{\alpha}}\sta{\oplus Av}{\longrightarrow}\mathcal{F}
$$
is surjective.
Consider the following commutative diagram,
\begin{eqnarray*}
\bigoplus_{\alpha\in S}{H^0(\mathcal{F}\otimes\alpha)\otimes H^0(H\otimes\check{\alpha})}\otimes\cO_A  &  \xrightarrow{\bigoplus{Av}} & H^0(\mathcal{F}\otimes H)^{\tilde{G}}\otimes\cO_A  \\
\downarrow  &  & \downarrow \\
\bigoplus_{\alpha\in S}{H^0(\mathcal{F}\otimes\alpha)\otimes H\otimes\check{\alpha}}= \bigoplus_{\alpha\in S}{H^0(\mathcal{F}\otimes\alpha)\otimes\check{\alpha}\otimes H} & \xrightarrow{Av\otimes id} & \mathcal{F}\otimes H. \\
\end{eqnarray*}
Then we have the surjectivity of the lower right map $Av\otimes id$.

We have to show surjectivity of the following evaluation map
$$
ev:H^0(\mathcal{F}\otimes H)^{\tilde{G}}\otimes\cO_A\rightarrow \mathcal{F}\otimes H.
$$ 
We first show that 
$$
\m{supp(coker(ev))} \subseteq \cap_{S\subset \hat{A}}{\lbrace\cup_{\alpha \in S}{B(H\otimes\check{\alpha})}\rbrace}=:Z.
$$
Here the intersection varies over finite subsets $S$ of $\hat{A}$ of large cardinality $N$ and
$B(H\otimes\check{\alpha})$ is the base locus of $H\otimes\check{\alpha}$. Let $x$ be an element in supp(coker$(ev)$) such that $x$ is not in $Z$.
This implies, for some $S$ and an $\al\in S$,
$$
H^0(H\otimes\check{\alpha})\otimes\cO_A \rightarrow H\otimes\check{\alpha}
$$
is surjective at $x$. Therefore, in the above commutative diagram, the evaluation map 
$$
ev:H^0(\mathcal{F}\otimes H)^{\tilde{G}}\otimes\cO_A\rightarrow \mathcal{F}\otimes H.
$$  is surjective at $x$. This gives a contradiction to $x$ lying in supp(coker$(ev)$). Hence supp(coker$(ev)$) $\subseteq \cap_{S\subset \hat{X}}{\lbrace\cup_{\alpha \in S}{B(H\otimes\check{\alpha})}\rbrace}$. Since $H$ is $G$- continuously globally generated, by the arguments in \cite[Remark 2.11, p.292]{PP}, $\cap_{S}\,\cup_{\alpha \in S}{B(H\otimes\check{\alpha})}$ is empty, where $\cap$ runs over $S\subset\hat{A}$ of large cardinality.
This implies supp(coker$(ev)$) is empty. 

Hence the evaluation map,
$$
ev:H^0(\mathcal{F}\otimes H)^{\tilde{G}}\otimes\cO_A\rightarrow \mathcal{F}\otimes H
$$ is surjective.

\end{proof}

The following proposition gives an analogue of \cite[Proposition 2.13]{PP}. It shows that the M-regularity of a coherent $G$-sheaf implies $G$-continuous global generation. We assume that the group $G$ acts freely on $A$.

\begin{proposition}\label{CGG}
If $\mathcal{F}$ is a M-regular coherent $G$-sheaf on $A$, then for any large positive integer $N$ and for any subset $S$ of $\hat{A}$ with cardinality $N$,
the sum of average maps,
$$
\bigoplus_{\alpha\in S}{H^0(\mathcal{F}\otimes\alpha)\otimes\check{\alpha}}\sta{\oplus Av}{\rightarrow}\mathcal{F}
$$
is surjective. In other words, $\cF$ is $G$-continuously globally generated.
\end{proposition}
\begin{proof}
Let $H$ be an ample $G$-line bundle such that $\mathcal{F}\otimes H$ is $G$-globally generated. Indeed, such a line bundle can be chosen, due to the correspondence in Proposition \ref{oneone}. We consider the sheaf $\cF_X$ corresponding to $\cF$, on $X=A/G$, and find an ample line bundle $H_X$ on $X$ such that $\cF_X\otimes H_X$ is globally generated on $X$. Let $H$ be the ample line bundle on $A$ corresonding to $H_X$. By Lemma \ref{GGoneone}, the coherent $G$-sheaf $\cF\otimes H$ is $G$- globally generated.

This implies that the evaluation map
$$
H^0(\mathcal{F}\otimes H)^{\tilde{G}}\otimes\cO_A \xrightarrow{ev} \mathcal{F}\otimes H
$$ is surjective.
Since $H$ is an ample $G$-line bundle, by Remark \ref{AM}, $H$ satisfies I.T with index $0$.
Therefore, by Corollary \ref{KTC}, 
$$
\bigoplus_{\alpha\in S}{H^0(\mathcal{F}\otimes\alpha)\otimes H^0(H\otimes\check{\alpha})}\otimes\cO_A\xrightarrow{\oplus{Av}} H^0(\mathcal{F}\otimes H)^{\tilde{G}}\otimes\cO_A 
$$ 
is surjective.
Now consider the following commutative diagram,
\begin{eqnarray*}
\bigoplus_{\alpha\in S}{H^0(\mathcal{F}\otimes\alpha)\otimes H^0(H\otimes\check{\alpha})}\otimes\cO_A  &  \xrightarrow{\bigoplus{Av}} & H^0(\mathcal{F}\otimes H)^{\tilde{G}}\otimes\cO_A  \\
\downarrow &  & \downarrow \\
\bigoplus_{\alpha\in S}{H^0(\mathcal{F}\otimes\alpha)\otimes H\otimes\check{\alpha}} & \xrightarrow{Av\otimes id} & \mathcal{F}\otimes H \\
\end{eqnarray*} 
where the sum varies over a finite subset $S$, of large cardinality. 
In the above commutative diagram, since $\oplus Av$ and the evaluation $ev$ are surjective, it follows that the averaging map
$$
\bigoplus_{\alpha\in S}{H^0(\mathcal{F}\otimes\alpha)\otimes H\otimes\check{\alpha}}\xrightarrow{Av\otimes id}\mathcal{F}\otimes H 
$$ 
is also surjective. Since $H$ is a line bundle, we obtain the assertion on $G$-continuous global generation of the sheaf $\cF$.
\end{proof}

As a consequence of the above proposition, we obtain the main result of this section:

\begin{corollary}\label{C1}
Suppose $\mathcal{F}$ is a coherent $G$-sheaf and $H$ is a $G$-line bundle on $A$. If both $\mathcal{F}$ and $H$ are M-regular sheaves on $A$,
then the coherent $G$-sheaf  $\mathcal{F}\otimes H$ is $G$-globally generated.
\end{corollary}
\begin{proof}
By Proposition \ref{CGG}, $ \mathcal{F}$ are $H$ are $G$-continously globally generated. By Proposition \ref{GG} $\mathcal{F}\otimes H$ is $G$-globally generated.
\end{proof}

\section{Embedding theorems on hyperelliptic varieties}\label{veryample}


In this section we prove analogues of very amplessnes results due to Ohbuchi and Lefschetz \cite[Corollary 3.9]{PP}, in the case of ample $G$-line bundle. By Corollary \ref{GGample}, we obtain similar embedding statements for the quotient variety $X=A/G$.

\begin{lemma}\label{VT}
Let $L_1$ and $L_2$ be $G$-line bundles on $A$ such that $L_1$ and $L_2\otimes I_{Gx}$ are $M$-regular, for all $a\in A$. Then $L_1\otimes L_2$ is
$G$-very ample on $A$.
\end{lemma}
\begin{proof}
By  Corollary \ref{C1},  $L_1\otimes L_2\otimes I_{G.a}$ is $G$-globally generated, for all $a\in A$. Hence $L_1\otimes L_2$ is $G$-very ample.
\end{proof}

Now we check $M$-regularity of $G$-line bundles which have no $G$-invariant base divisor. This will enable us to conclude very ampleness of powers of $G$-line bundles.

\begin{proposition}\label{VT1}
Suppose $L$ be an ample $G$-line bundle and having no $G$-invariant base divisor on an abelian variety $A$. Then $L\otimes I_{G.a}$ is $M$-regular on $A$.
\end{proposition}
\begin{proof}
Firstly for any $a\in A$, consider the following exact sequence:
$$
0 \rightarrow {L\otimes I_{G.a}}\rightarrow {L} \rightarrow L_{|G.a} \rightarrow 0.
$$
Take the long exact cohomology sequence:
$$
0 \rightarrow H^0(L\otimes I_{G.a})\rightarrow H^0(L) \rightarrow \oplus_{g\in G}H^0(L\otimes \C(ga)) \rightarrow
$$
$$
H^1(L\otimes I_{G.a})\rightarrow H^1(L) \rightarrow \oplus_{g\in G}H^1(L\otimes \C(ga)) \rightarrow \cdot \cdot \cdot.
$$
Also note that since $L$ is ample $H^i(A, L)=0$, for all $i>0$. Therefore the above long exact sequence reduces to
$$
0 \rightarrow H^0(L\otimes I_{G.a})\rightarrow H^0(L) \rightarrow (\oplus_{g\in G}H^0(L\otimes \C(ga)) \rightarrow H^1(L\otimes I_{G.a})\rightarrow 0. 
$$

Now consider the cohomological support locus, 
$$
\m{Supp }V^i(L\otimes I_{G.a}):= \lbrace \alpha \in \hat{A} : H^i(L\otimes I_{G.a} \otimes \alpha)\neq 0\rbrace.
$$ 
Note that
$$
L\otimes I_{G.a} \otimes \alpha =\oplus_{g\in G}(L\otimes I_{ga}\otimes \alpha)\cong t^*_{y}(L\otimes I_{G.a-y}),
$$
for some $y\in A$.
The above exact sequences imply that, when $i>1$, we have  Supp$V^i(L\otimes I_{Gx}))=\emptyset $. This implies 
$$
\m{codim Supp}V^i(L\otimes I_{Gx})\,>\,i
$$ 
for all $i>1$.
When $i=1$, Supp$(V^1(L\otimes I_{Gx}))$ is isomorphic to a $G$-invariant base locus of $L$. By hypothesis, $L$ has no $G$-invariant base divisor. Hence this implies codimension of Supp$(V^1(L\otimes I_{Gx}))$
is at least $2$. Hence, using \eqref{cohlocus}, $L\otimes I_{Gx}$ is M-regular.

\end{proof}

Now we consider powers of ample $G$-line bundles and apply the previous results to obtain embedding statements.

\begin{theorem}\label{VT2}
Suppose $N$ is an ample line bundle on the quotient variety $X=A/G$. Then the following hold:

a) $N^2$ is very ample, if $N$ has no base divisor.

b) $N^3$ is always very ample.
\end{theorem}
\begin{proof}
Using Proposition \ref{oneone}, let $L$ be the ample $G$-line bundle on $A$ corresponding to the ample line bundle $N$ on $X$. 

To prove a), we assume that $N$ has no base divisor. This implies that $L$ has no $G$-invariant base divisor.
By Proposition \ref{VT1}, $L\otimes I_{Gx}$ is M-regular, for all $x\in X$. Furthermore since $L$ is ample, $L$ is M-regular by Remark \ref{AM}. Hence by Corollary \ref{C1}, $L\otimes L\otimes I_{Gx}$ is $G$-globally generated. Hence $L^{\otimes 2}$ is $G$-very ample.
Now by Corollary \ref{GGample}, we conclude that $N^2$ is very ample on $X$.

To prove b), note that by Corollary \ref{C1}, $L^{\otimes 2}$ is $G$-globally generated. This implies that $L^{\otimes 2}$
has no base divisor and hence by Theorem \ref{VT1}, $L^{\otimes 2}\otimes I_{Gx}$ is M-regular, for all $x\in X$. Hence, by Corollary \ref{C1}, $L^{\otimes 2}\otimes I_{Gx}$ is $G$-continuously globally generated.
This implies $L^{\otimes 3}$ is $G$-very ample and hence $N^{\otimes 3}$ is very ample on $X$.
\end{proof}

To extend above results to $k$-jet ampleness on a hyperelliptic vareity $X$, we note the below lemma for ample $G$-line bundles on an abelian variety $A$. 

\begin{lemma}\label{lemmakjet}
Suppose $L$ is an ample $G$-line bundle on an abelian variety $A$. Then the following are equivalent:

1) $L$ is $G$-$k$-jet ample.

2) $L\otimes I_{G.a_1}^{k_1}\otimes...\otimes I^{k_l}_{G.a_l}$ satisfies I.T. with index $0$, for any $l$-distinct points $a_1,...,a_l\in A$ such that $\sum k_i=k+1$.

3)  $L\otimes I_{G.a_1}^{k_1}\otimes...\otimes I^{k_l}_{G.a_l}$ is $G$-globally generated, for any $l$-distinct points $a_1,...,a_l\in A$ such that $\sum k_i=k$.
\end{lemma}
\begin{proof}
Using the correspondence in Proposition \ref{oneone}, it suffices to prove the equivalence for the corresponding line bundle $N:=\pi_*(L)^{G}$ on $X$. Recall that $\pi:A\rar X=A/G$ is the quotient morphism.
Using \eqref{directsum}, we note that
$$
H^1(A,L)=\bigoplus_{\chi\in \hat{G}}H^1(X,N\otimes L_\chi).
$$
Here $L_\chi$ denotes the line bundle on $X$ associated to the character $\chi$ on $G$.
Since $L$ is ample we have the vanishing $H^1(A,L)=0$. This implies the vanishing $H^1(X,N)=0$. 
The rest of the proof is similar to \cite[Lemma 3.3]{PP2}, and we omit it.
\end{proof}

Now we state the analogue of above theorem, for higher jet ampleness on a hyperelliptic variety $X$.

\begin{proposition}
Suppose $N$ is an ample line bundle on a hyperelliptic variety $X$. 
Then the following hold:

1)  $N^{k+1}$ is $k$-jet ample if $N$ has no base divisor, and for $k\geq 1$.

2) $N^{k+2}$ is $k$-jet ample, and for $k\geq 0$.
\end{proposition}
\begin{proof}
The proof is similar to \cite[Theorem 3.8]{PP2} applied to the corresponding ample $G$-line bundle $L$ on $A$. Indeed, by above Lemma \ref{lemmakjet}, it suffices to check 3), i.e., the sheaf
$$
L\otimes I_{G.a_1}^{k_1}\otimes...\otimes I^{k_l}_{G.a_l}
$$ 
is $G$-globally generated, for any $l$-distinct points $a_1,...,a_l\in A$ such that $\sum k_i=k$.

We apply induction on $k$, and using the correspondence in Corollary \ref{GGample}, prove it for the ample $G$-line budle $L$ on $A$. 

Suppose $k=1$. Then 1) holds, by Theorem \ref{VT2}. 

Suppose the statement 1) holds for $L^{k-1}$, i.e., $L^{k-1}$ is $G$-$k$-jet ample. By Lemma \ref{lemmakjet}, this implies for any $l$-distinct points $a_1,...,a_l\in A$ such that $\sum_ik_i=k$, the sheaf $L^k\otimes I_{G.a_1}^{k_1}\otimes...\otimes I_{G.a_l}^{k_l}$ satisfies I.T with index zero. 
By Remark \ref{AM} 2), $L^k\otimes I_{G.a_1}^{k_1}\otimes...\otimes I_{G.a_l}^{k_l}$ is $M$-regular. Hence, by Corollary \ref{C1}, the sheaf $L\otimes L^k\otimes I_{G.a_1}^{k_1}\otimes...\otimes I_{G.a_l}^{k_l}$ is $G$-globally generated, for $l$-distinct $a_1,...,a_l \in A$, such that $\sum k_i=k$. Now by Lemma \ref{lemmakjet} 3), $L^{k+1}$ is $G$-$k$-jet ample.    

The proof of 2) is similar, and we omit it.
\end{proof}

\section{Syzygy or $N_p$-property of line bundles on a hyperelliptic variety}\label{syzygy}


In this section, we look at syzygy or $N_p$-properties defined by M. Green \cite{Green}.

Suppose $Z$ is a smooth projective variety defined over the complex numbers.
An ample line bundle $L$ on $Z$ is said to satisfy $N_p$-property if the first $p$-steps of the minimal graded free resolution of the algebra $R_L:= \oplus_{n\geq 0} H^0(L^n)$ over the polynomial ring $S_L:=\oplus_{n\geq 0}Sym^n H^0(L)$ are linear. In other words, a minimal resolution of $R_L$ looks like:
$$
S_L(-p-1)^{i_p}\rar S_L(-p)^{\oplus i_{p-1}} \rar ...\rar S_L(-2)^{i_1}\rar S_L\rar R_L\rar 0.
$$
When $p=0$, we say that $L$ gives a projectively normal embedding. When $p=1$, $L$ satisfies $N_0$ and the ideal of the embedded variety is generated by quadrics.

\subsection{Criterion for $N_p$-property}

Usually, in practice, one looks at surjectivity of multiplication maps of sections of some natural bundles associated to $L$. We recall them below.
Consider the exact sequence associated to a globally generated line bundle $L$, given by evaluation of its sections:
$$
0\rar M_L\rar H^0(L)\otimes \cO_Z\rar L\rar 0.
$$
Here $M_L$ is a coherent sheaf and is the kernel of the evalution map. In fact, it is a locally free sheaf.

Consider the exact sequence by taking the $p+1$-st exterior power of the above evaluation sequence:
$$
0\rar \wedge^{p+1}M_L\otimes L^h \rar \wedge^{p+1}H^0(L)\otimes L^h \rar \wedge^pM_L\otimes L^{h+1}\rar 0.
$$
Then $N_p$-property holds if 
$$
H^1(\wedge^{p+1}M_L\otimes L^h)=0, \m{ for all } h\geq 1.
$$
The converse is true if $Z$ is an abelian variety, since $H^1(L^h)=0$.
See  \cite[p.660]{Pareschi}.
Moreover we have:

\begin{lemma}
a) If $H^1(\wedge^{p+1}M_L\otimes L^h)=0, \m{ for all } h\geq 1$, then $L$ satisfies $N_p$-property.

b) Assume that $H^1(\wedge^{p}M_L\otimes L^h)=0$ for $h\geq 1$. Then $H^1(\wedge^{p+1}M_L\otimes L^h)=0$ if and only if the multiplication map
$$
H^0(L)\otimes H^0(M_L^{\otimes p}\otimes L^h) \rar H^0(M_L^{\otimes p}\otimes L^{\otimes h+1})
$$ 
is surjective.
\end{lemma}
\begin{proof}
See \cite[Lemma 4.1]{Pareschi}.
\end{proof}

\subsection{Cohomology Vanishing on a hyperelliptic variety} 

Suppose $X$ is a hyperelliptic variety of dimension $n$. As in earlier sections, we consider the quotient morphism
$\pi:A\rar X=A/G$. Here $G$ is a finite group acting freely on $A$.

Suppose $N$ is an ample line bundle on $X$. Assume it is globally generated. Consider the evaluation map on the sections of $N$:
$$
0\rar M_N\rar H^0(N)\otimes \cO_X \rar N \rar 0.
$$
Pullback of this exact sequence on $A$ yields the exact sequence:
$$
0\rar \pi^*M_N \rar H^0(L)^{\tilde{G}}\otimes \cO_A \rar L\rar 0.
$$
Here $L:=\pi^*N$ is the corresponding $G$-line bundle on $A$, and $H^0(L)^{\tilde{G}}\subset H^0(L)$ is the subspace of $\tilde{G}$-invariant sections. Denote $M_L^G:=\pi^*M_N$. In particluar, $\wedge^pM_L^G$ is a $G$-linearized bundle.

We first note the below vanishing, which we will need.

\begin{lemma}\label{vanish1}
The cohomology vanishing 
$$
H^1(A, \wedge^{p+1}M_L^G\otimes L^h)=0
$$
implies the cohomology vanishing
$$
H^1(X, \wedge^{p+1}M_N \otimes N^h)=0,
$$
for each $h\geq 1$.
\end{lemma}
\begin{proof}
Since the bundles $\wedge^{p+1}M_L^G$ and $L^h$ are $G$-linearized bundles, the tensor product $\wedge^{p+1}M_L^G\otimes L^h$ is also a $G$-linearized bundle. In particular, the group $\tilde{G}$ acts on the cohomology groups $H^i(A,\wedge^{p+1}M_L^G\otimes L^h)$, for $i\geq 0$. 
The $\tilde{G}$-invariant subspace is precisely $H^i(A, \wedge^{p+1}M_L^G\otimes L^h)^{\tilde{G}}$. Now, we use projection formula as shown in Lemma \ref{GGoneone}, and using \eqref{directsum}, we deduce that
the $\tilde{G}$-invariant subspace is equal to the cohomology group $H^i(X, \wedge^{p+1}M_N \otimes N^h)$.
This gives the assertion. 

\end{proof}

\begin{lemma}\label{vanish2}
The cohomology vanishing
$$
 H^1(A, \wedge^{p+1}M_L\otimes L^h)=0
$$
implies the cohomology vanishing
$$
H^1(X, \wedge^{p+1}M_L^G \otimes L^h)=0,
$$
for each $h\geq 1$.
\end{lemma}
\begin{proof}
Note that in the below exact sequence 
$$
0\rar M_L\rar H^0(L)\otimes \cO_A \rar L\rar 0
$$
the group $\tilde{G}$ acts on $H^0(L)$ and on $L$ equivariantly. Hence the inclusion of $\tilde{G}$-invariant sections $H^0(L)^{\tilde{G}}\subset H^0(L)$ provides an inclusion of bundles
$$
M_L^G\subset M_L.
$$
Moreover, since the averaging map of sections
$$
H^0(L)\sta{Av}{\rar} H^0(L)^{\tilde{G}} ,\,\,\,\,\,s\mapsto \f{1}{|G|}\sum_{g\in \tilde{G}}g.s
$$ 
is surjective, we deduce that the bundle $M_L^G$ is a split summand of $M_L$.

Hence we have an inclusion of their exterior powers tensored with $L^h$:
$$
\wedge^{p+1}M_L^G \otimes L^h \subset \wedge^{p+1}M_L \otimes L^h.
$$
This is also a split summand and hence gives the inclusion of cohomologies:
 $$
 H^1(A, \wedge^{p+1}M_L^G \otimes L^h) \subset H^1(A, \wedge^{p+1}M_L \otimes L^h).
$$
We now deduce our assertion.

\end{proof}

Now, we apply above two lemmas to conclude our main consequence of this section.

\begin{proposition}\label{Nphyp}
Suppose $M$ is an ample line bundle on a hyperelliptic variety $X$. 
Then $M^{p+k}$ satisfies $N_p$-property, for any $k\geq 3$.
\end{proposition}
\begin{proof}
Suppose $M$ is an ample line bundle on $X$.
By Theorem \ref{VT2}, we know that $N:=M^k$, $k\geq 3$, is very ample. In particular, $N$ is globally generated.
Since $L=\pi^*N$ is an ample globally generated line bundle on $A$, by \cite[Theorem 4.3, p. 663]{Pareschi}, we have the cohomology vanishing
$$
H^1(A, \wedge^{p+1}M_L \otimes L^h)=0
$$
for any $h\geq 1$. Now apply Lemma \ref{vanish1} and Lemma \ref{vanish2}, to conclude the cohomology vanishing
$$
H^1(X, \wedge^{p+1}M_N \otimes N^h)=0.
$$
for any $h\geq 1$.
This implies that $N^p=M^{p+k}$, $k\geq 3$, satisfies $N_p$-property. 

\end{proof}




\begin{thebibliography}{AAAAAA}


\bibitem [Bg-dF]{Bg-dF} G. Bagnera and M. de Franchis : {\em Sopra le superficie algebrique de hanno le coordintae det punto generico esprimibili con funzioni meromorfe quadruplamente periodiche di due parametri}, Rend. della Reale Accad.dei Linci, Ser, V, XVI (1907), 492--498.

\bibitem[Bn-Ga]{Gallego} P. Bangere and F. Gallego : {\em Projective Normality and Syzygies of Algebraic Surfaces}, Journal f\"ur reine und angewandte Mathematik, 506, 145-180, 1999.

\bibitem[En-Se]{En-Se} F. Enriques and F. Severi : {\em  M\'emoire sur les surfaces hyperelliptiques.} (French)  Acta Math.  32  (1909),  no. 1, 283--392. 

\bibitem[BL]{BL} C. Birkenhake and H. Lange : {\em Complex Abelian Varieties}, A Series of Comprehensive Studies in Math. \textbf{302}, Springer-Verlag, New York,
2003.

\bibitem [Ei-Lz]{EL} L. Ein and R. Lazarsfeld : {\em Syzygies and Koszul cohomology of smooth projective varieties of arbitrary dimension.}  Invent. Math.  111  (1993),  no. \textbf{1}, 51--67. 

\bibitem [Gr]{Green} M. Green : {\em Koszul cohomology and the geometry of projective varieties.}  J. Differential Geom.  19  (1984),  no. \textbf{1}, 125--171.

\bibitem[Gr-Lz]{GreenLaz} M. Green, R. Lazarsfeld : {\em On the projective normality of complete linear series on an algebraic curve.}  Invent. Math.  83  (1985),  no. \textbf{1}, 73--90.

\bibitem [Gr-Lz2]{Green2} M. Green and R. Lazarsfeld : {\em Deformation theory, generic vanishing theorems, and some conjectures of Enriques, Catanese and Beauville.}  Invent. Math.  90  (1987),  no. \textbf{2}, 389--407. 

\bibitem [Iy]{Iyer} J.N. Iyer : {\em Projective normality of abelian varieties.}  Trans. Amer. Math. Soc.  355  (2003),  no. \textbf{8}, 3209--3216.

\bibitem [Ke]{Kempf} G. Kempf : {\em Linear systems on abelian varieties.}  Amer. J. Math.  111  (1989),  no. \textbf{1}, 65--94.

\bibitem [Ko]{Koizumi} S. Koizumi : {\em Theta relations and projective normality of abelian varieties},
 Amer. Jour. of Math., \textbf{98}, 865-889, 1976.
 
\bibitem[La]{Lange} H. Lange: {\em Hyperelliptic varieties}, Tohoku Math. J. (2)  53  (2001),  no. \textbf{4}, 491510.

\bibitem[La2]{Lange2} H. Lange and S. Recillas : {\em  Abelian varieties with group action.}  J. Reine Angew. Math.  \textbf{575}  (2004), 135--155.

\bibitem[Le]{Lefschetz} S. Lefschetz: {\em Hyperelliptic surfaces and abelian varieties}. in Selected topics in Algebraic geometry I, 349--395, 1928.

\bibitem[Mk]{M} S. Mukai : {\em Duality between $D(X)$ and $D(\hat{X})$ with application to Picard sheaves}, Nagoya Math. J., \textbf{81}, 153-175, 1981.

\bibitem[Mu]{Mumford} D. Mumford : {\em Abelian varieties}, Tata Institute of Fundamental Research Studies in Mathematics, No. 5  Published for the Tata Institute of Fundamental Research, Bombay; Oxford University Press, London 1970 viii+242 pp.

\bibitem[MFK]{MFK} D. Mumford, J. Fogarty and F. Kirwan :  {\em Geometric invariant theorey}.
Third enlarged edition, Springer, 1994.

\bibitem[Oh]{Ohbuchi} A. Ohbuchi : {\em A note on the normal generation of ample line bundles on abelian varieties.}  Proc. Japan Acad. Ser. A Math. Sci.  64  (1988),  no. \textbf{4}, 119--120.

\bibitem[Pa]{Pareschi} G. Pareschi :{\em Syzygies of abelian varieties},  J. Amer. Math. Soc.  13  (2000),  no. \textbf{3}, 651--664. 

\bibitem[Pa-Po]{PP} G. Pareschi and M.Popa : {\em Regularity on Abelian Varieties I}, J. Amer. Math. Soc., \textbf{16}, 285-302, 2003.

\bibitem[Pa-Po2]{PP2} G. Pareschi and M. Popa : {\em Regularity on Abelian Varieties II: basic results on linear series and defining equations}, J. Algebraic Geom. 13 (2004), 167-193. 

\bibitem[Vo]{Voisin} C. Voisin : {\em Green's canonical syzygy conjecture for generic curves of odd genus.}  Compos. Math.  141  (2005),  no. \textbf{5}, 1163--1190. 



\end{thebibliography}
\end{document}